\documentclass[12pt]{amsart}
\usepackage{times,fullpage}
\usepackage{latexsym,amscd,amssymb}
\newtheorem{thm}{Theorem}
\newtheorem{prop}[thm]{Proposition}
\newtheorem{lem}[thm]{Lemma}
\newtheorem{cor}[thm]{Corollary}

\theoremstyle{remark}
\newtheorem{rem}[thm]{Remark}

\theoremstyle{definition}
\newtheorem{defn}[thm]{Definition}

\newcommand{\bC}{\mathbb C}

\newcommand{\blowup}{\overline{\bC P^2}}

\newcommand{\C}{\mathbb{ C}}

\newcommand{\OO}{\mathcal{ O}}

\newcommand{\Q}{\mathbb{ Q}}

\newcommand{\PP}{\mathbb{ P}}

\newcommand{\DO}{{\mathcal DO}}
\newcommand{\D}{{\mathcal D}}
\newcommand{\HH}{{\mathcal H}}
\newcommand{\E}{{\mathcal E}}
\newcommand{\EE}{{\mathbb E}}
\newcommand{\F}{{\mathcal F}}
\newcommand{\I}{{\mathcal I}}
\newcommand{\LL}{{\mathcal L}}

\title{Topologically invariant Chern numbers of projective varieties}
\author{D.~Kotschick}
\address{Mathematisches Institut, {\smaller LMU} M\"unchen,
Theresienstr.~39, 80333~M\"unchen, Germany}
\email{dieter@member.ams.org}
\date{November 1, 2009; \copyright{\ D.~Kotschick 2009}}
\subjclass[2000]{primary 57R20, 57R77; secondary 14J99, 55N22}
\keywords{Chern numbers, complex cobordism, Hirzebruch problem}

\begin{document}

\begin{abstract}
We prove that a rational linear combination of Chern numbers is an oriented diffeomorphism invariant of smooth complex
projective varieties if and only if it is a linear combination of the Euler and Pontryagin numbers. In dimension at least three we 
prove that only multiples of the top Chern number, which is the Euler characteristic, are invariant under diffeomorphisms 
that are not necessarily orientation-preserving. These results solve a long-standing problem of Hirzebruch's.
We also determine the linear combinations of Chern numbers that can be bounded in terms of Betti numbers.
\end{abstract}

\maketitle


\section{Introduction}

\subsection*{Statement of results}

In 1954, Hirzebruch~\cite[Problem~31]{Hir1} asked which linear combinations of Chern numbers of smooth
complex projective varieties are topologically invariant. 
The purpose of this paper is to provide complete answers to this question. Of course,
the answers depend on what exactly one means by topological invariance.

Since the manifold underlying a complex-algebraic variety has a preferred orientation, it is most natural to examine
which linear combinations of Chern numbers are invariant under orientation-preserving homeo- or diffeomorphisms.
The answer is given by:
\begin{thm}\label{t:DO}
A rational linear combination of Chern numbers is an oriented diffeomorphism invariant of smooth complex projective 
varieties if and only if it is a linear combination of the Euler and Pontryagin numbers. 
\end{thm}
In one direction, the Euler number, which is the top Chern number, is of course a homotopy invariant.
Further, the Pontryagin numbers, which are special linear combinations of Chern numbers, are oriented
diffeomorphism invariants. In fact, Novikov~\cite{N} proved that the Pontryagin numbers are also invariant under 
orientation-preserving homeomorphisms, and so Theorem~\ref{t:DO} is unchanged if we replace 
oriented diffeomorphism-invariance by oriented homeomorphism-invariance.
The other direction, proving that there are no other linear combinations that are oriented diffeo\-mor\-phism-invariants,
has proved to be quite difficult because of the scarcity of examples of diffeomorphic projective
varieties with distinct Chern numbers.

Given Theorem~\ref{t:DO}, and the fact that Pontryagin numbers depend on the orientation, one might expect that
only the Euler number is invariant under homeo- or diffeomorphisms that do not necessarily preserve the 
orientation. For diffeomorphisms this is almost but not quite true:
\begin{thm}\label{t:D}
In complex dimension $n\geq 3$ a rational linear combination of Chern numbers is a diffeomorphism invariant 
of smooth complex projective varieties if and only if it is a multiple of the Euler number $c_n$. In complex dimension
$2$ both Chern numbers $c_2$ and $c_1^2$ are diffeomorphism-invariants of complex projective surfaces.
\end{thm}
The statement about complex dimension $2$ is a consequence of Seiberg--Witten theory and was first proved 
in~\cite[Theorem~2]{BLMSorient}; see also~\cite[Theorem~1]{Chern}. It is an exception due to the special nature of 
four-dimensional differential topology. The exception disappears if we consider homeomorphisms instead of diffeomorphisms:
\begin{thm}\label{t:H}
A rational linear combination of Chern numbers is a homeomorphism invariant 
of smooth complex projective varieties if and only if it is a multiple of the Euler number. 
\end{thm}

These theorems show that linear combinations of Chern numbers of complex projective varieties are not usually determined
by the underlying manifold. This motivates the investigation of a modification of Hirzebruch's original problem, asking how far 
this indeterminacy goes. More precisely, we would like to know which linear combinations of Chern numbers are determined 
up to finite ambiguity by the topology. The obvious examples for which this is true are the numbers
$$
\chi_{p} = \chi (\Omega^p)= \sum_{q=0}^{n} (-1)^{q}h^{p,q} \ ,
$$
where $n$ is the complex dimension.
By the Hirzebruch--Riemann--Roch theorem~\cite{TMAG} these are indeed linear combinations of Chern numbers.
By their very definition, together with the Hodge decomposition of the cohomology, the numbers $\chi_p$ are bounded
above and below by linear combinations of Betti numbers. It turns out that this property characterizes the linear 
combinations of the $\chi_p$, as shown by the following theorem:
\begin{thm}\label{t:Betti}
A rational linear combination of Chern numbers of smooth complex projective varieties can be bounded in terms of Betti 
numbers if and only if it is a linear combination of the $\chi_p \ $. 
\end{thm}
As a consequence of this result, most linear combinations of Chern numbers are independent of the Hodge structure:
\begin{cor}\label{c:Hodge}
A rational linear combination of Chern numbers of smooth complex projective  varieties is determined by the Hodge  
numbers if and only if it is a linear combination of the $\chi_p \ $. 
\end{cor}
The span of the $\chi_p$ includes the Euler number $c_n = \sum_p (-1)^p\chi_p$
and the signature, which, according to the Hodge index theorem, equals $\sum_p \chi_p$.
It also includes the Chern number $c_1c_{n-1}$, by a result of Libgober and Wood~\cite[Theorem~3]{LW}.
Nevertheless, the span of the $\chi_p$ is a very small subspace of the space of linear combinations of Chern
numbers. The latter has dimension equal to $\pi(n)$, the number of partitions of $n$, which grows exponentially 
with $n$. The former has dimension $[(n+2)/2]$, the integral part of $(n+2)/2$. This follows from the symmetries 
of the Hodge decomposition, which imply $\chi_p = (-1)^n\chi_{n-p}$, together with the fact that 
$\chi_0,\ldots,\chi_{[n/2]}$ are linearly independent, as can be easily checked by evaluating on products of 
projective spaces.

Hirzebruch's problem is of central importance in the applications of the Hirzebruch--Riemann--Roch formula, for example in the 
classification theory of algebraic varieties. The right-hand side of HRR is given by the $\chi_p$, and, by the results of this paper,
this right-hand-side is not usually a topological invariant. In fact, as explained in~\cite[Theorem~6]{PNAS}, the only linear combinations
of the $\chi_p$ that are also linear combinations of Euler and Pontryagin numbers are the multiples of the signature. 

\subsection*{History and outline}

The first example of a pair of diffeomorphic projective varieties with distinct Chern numbers was found by Borel and Hirzebruch~\cite{BH} in 1959. 
Using Lie theory, they showed that the homogeneous space $SU(4)/S(U(2)\times U(1)\times U(1))$ has two invariant structures as a Hodge 
manifold with different values for the Chern number $c_1^5$. At the time this may well have seemed to be some sort of isolated curiosity. 
It was only comparatively recently that Calabi and Hirzebruch~\cite{H05} clarified the geometric meaning of the example by identifying the two 
structures as the projectivised holomorphic tangent and cotangent bundles of $\C P^3$. Terzi\'c and I~\cite{KT} generalized this
example to $SU(n+2)/S(U(n)\times U(1)\times U(1))$, where the two structures correspond to the projectivised holomorphic tangent and 
cotangent bundles of $\C P^{n+1}$. For all $n\geq 2$, the Chern numbers $c_1^{2n+1}$, and many others, are different for the two
structures, although the underlying smooth manifold is the same, and the Hodge numbers agree~\cite{KT}.

Over the years a number of attempts have been made to find further examples of diffeomorphic varieties with distinct Chern
numbers. However, until recently they were all unsuccessful.

It turns out that a key ingredient for the proofs of Theorems~\ref{t:DO}, \ref{t:D} and~\ref{t:H}, and thus for the solution of 
Hirzebruch's problem, is the existence of homeomorphic simply connected algebraic surfaces with different signatures, equivalently
with different $c_1^2$. The existence of such pairs of surfaces, which I proved in~\cite{MAorient}, depends on deep results in both 
surface geography and four-dimensional topology. Of course, the homeomorphism in question cannot preserve the orientation.
By the main result of~\cite{BLMSorient}, the homeomorphism is not smoothable.
In Theorem~\ref{t:MA} below we recall the main result of~\cite{MAorient} in the form we shall use it here.

In~\cite{Chern}, I used the examples from~\cite{MAorient} as building blocks to manufacture certain pairs of diffeomorphic three-folds
and four-folds, for which I could compute all the Chern numbers explicitly. Although this was sufficient to prove Theorem~\ref{t:DO} in 
complex dimensions $\leq 4$, it is clear that such a brute force approach cannot work in general. In this paper we use cobordism theory 
to minimize the need for concrete calculations. This approach is inspired by the work of Kahn~\cite{K2}, who used a similar strategy to 
prove the analogue of Theorem~\ref{t:DO} for almost complex manifolds in place of algebraic varieties. Of course
the implementation of the strategy is a lot easier in that case, because almost complex structures are much more flexible
and exist in abundance. Essentially we shall have to find only one pair of examples with specific properties in each dimension. In 
complex dimension $n$ the examples will be algebraic $\C P^{n-2}$-bundles over the orientation-reversingly homeomorphic 
algebraic surfaces found in~\cite{MAorient}.

The proof of Theorem~\ref{t:Betti} is independent of the results of~\cite{MAorient}, but, like the proofs of the other main theorems, 
it also uses cobordism theory. However,
the way in which cobordism theory is used in that proof is new; there is no parallel for that argument in earlier work.
The results of the present paper were announced in~\cite{PNAS}, which also contains some related observations.

\subsection*{Acknowledgments}
I would like to thank D.~Toledo for bringing the work of Kahn~\cite{K2} to my attention.
I gratefully acknowledge the support of The Bell Companies Fellowship at the Institute for Advanced Study in Princeton.

\section{Complex cobordism}\label{s:cob}

In the proofs of the main theorems we shall make use of the rational complex cobordism ring $\Omega^U_{\star}\otimes\Q$.
(We use the classical terminology, calling this a {\it cobordism ring}, rather than the newer convention referring to $\Omega^U_{\star}\otimes\Q$
as a {\it bordism ring}.)
This ring is a polynomial ring with one generator $\beta_i$ in every even degree $2i$; see~\cite{M,NN,TM} and also~\cite{MDT,BuonCri}.
Two stably almost complex manifolds of the same dimension have the same Chern numbers if and only if they represent the 
same class in $\Omega^U_{\star}\otimes\Q$. Equivalently, in complex dimension $i$ the vector space of rational linear combinations 
of Chern numbers is the dual space of $\Omega^U_{2i}\otimes\Q$.

We define three ideals in $\Omega^U_{\star}\otimes\Q$, each one corresponding to the situation considered in one of the first three
main theorems.
The first definition is:
\begin{defn}\label{d:DO}
Let $\DO_{2i}\subset \Omega^U_{2i}\otimes\Q$ be the linear subspace spanned by the set
$$
\{ [M]-[N]\in\Omega^U_{2i}\otimes\Q \ \vert \ M, \ N \  \textrm{orientation-preservingly} \ \textrm{diffeomorphic} \ \textrm{projective} \ \textrm{varieties}\}
$$ 
of differences of complex projective varieties of complex dimension $i$ that are orientation-pre\-ser\-ving\-ly diffeomorphic.
\end{defn}
Clearly the direct sum of the $\DO_{2i}$ is an ideal $\DO_{\star}$ in the ring $\Omega^U_{\star}\otimes\Q$. 
We shall prove the following:
\begin{thm}\label{t:DOc}
There is a sequence of ring generators $\beta_1, \beta_2, \beta_3, \ldots$ for the rational complex cobordism
ring $\Omega^U_{\star}\otimes\Q$ with the property that for each odd index $i\geq 3$ the elements $\beta_1\cdot\beta_{i-1}$
and $\beta_i$ are contained in $\DO_{2i}$. 
\end{thm}
We now explain how this result implies Theorem~\ref{t:DO}. The vector space of rational linear combinations of Chern numbers that
are oriented diffeomorphism-invariants of complex projective varieties of complex dimension $i$ is the annihilator of the subspace
$\DO_{2i}\subset\Omega^U_{2i}$. Therefore, its dimension equals the codimension of $\DO_{2i}$ in $\Omega^U_{2i}$. Now 
Theorem~\ref{t:DOc} implies that a monomial of degree $i$ in the $\beta_j$ is in $\DO_{2i}$ if it contains a $\beta_j$ with $j$ odd and $j>1$,
or if it contains $\beta_1$ and a $\beta_j$ with $j$ even. If $i$ is odd, then the only monomial that is not obviously in $\DO_{2i}$ is 
$\beta_1^i$. This means that the codimension of $\DO_{2i}$ in $\Omega^U_{2i}\otimes\Q$ is at most one, and so the only linear 
combinations of Chern numbers invariant under orientation-preserving diffeomorphism 
are the multiples of the Euler number $c_i$. If $i$ is even, then, in addition to $\beta_1^i$, all the monomials containing only $\beta_{j_l}$
with all $j_l$ even may be outside $\DO_{2i}$. The number of these other monomials is exactly $\pi(i/2)$, the number of partitions of $i/2$.
This is also the number of Pontryagin numbers in dimension $2i$. This completes the deduction of Theorem~\ref{t:DO} from Theorem~\ref{t:DOc}.

To formulate cobordism statements that will imply Theorems~\ref{t:D} and~\ref{t:H} we require the following:
\begin{defn}
Let $\D_{2i}\subset \Omega^U_{2i}\otimes\Q$ be the linear subspace spanned by the set
$$
\{ [M]-[N]\in\Omega^U_{2i}\otimes\Q \ \vert \ M, \ N \   \textrm{diffeomorphic} \ \textrm{projective} \ \textrm{varieties}\}
$$ 
of differences of complex projective varieties of complex dimension $i$ that are diffeomorphic,
and $\HH_{2i}\subset \Omega^U_{2i}\otimes\Q$ be the linear subspace spanned by the set
$$
\{ [M]-[N]\in\Omega^U_{2i}\otimes\Q \ \vert \ M, \ N \   \textrm{homeomorphic} \ \textrm{projective} \ \textrm{varieties}\}
$$ 
of differences of complex projective varieties of complex dimension $i$ that are homeomorphic,
in both cases without any assumption about orientations.
\end{defn}
Again the direct sums of the $\D_{2i}$, respectively of the $\HH_{2i}$, define an ideal $\D_{\star}$, respectively $\HH_{\star}$,
in the ring $\Omega^U_{\star}\otimes\Q$.
We shall prove:
\begin{thm}\label{t:Dc}
There is a sequence of ring generators $\beta_1, \beta_2, \beta_3, \ldots$ for the rational complex cobordism
ring $\Omega^U_{\star}\otimes\Q$ with the properties $\beta_1\cdot\beta_2\in\D_6$, $\beta_2\cdot\beta_2\in\D_8$ 
and $\beta_i\in\D_{2i}$ for all $i\geq 3$. 
\end{thm}
\begin{thm}\label{t:Hc}
There is a sequence of ring generators $\beta_1, \beta_2, \beta_3, \ldots$ for the rational complex cobordism
ring $\Omega^U_{\star}\otimes\Q$ with the property $\beta_i\in\HH_{2i}$ for all $i\geq 2$. 
\end{thm}
Note that, because the Chern numbers of algebraic surfaces are diffeomorphism-invariant~\cite{BLMSorient,Chern}, it is not 
possible to choose the generator $\beta_2$ to be in $\D_4$.

Dimension counts similar to the one above show that Theorem~\ref{t:H} and the case of complex dimension $\geq 3$ in
Theorem~\ref{t:D} follow from  Theorem~\ref{t:Hc} and Theorem~\ref{t:Dc} respectively. 


The proof that certain differences of projective varieties can be taken as generators uses Milnor's characterization 
of ring generators for $\Omega^U_{\star}\otimes\Q$.
Let $M$ be a closed almost complex manifold of real dimension $2k$ with total Chern class
$$
c(TM)=\prod_i (1+x_i) \  .
$$
Then, following Thom~\cite{T}, one defines the number $s_k(M)$ as 
$$
s_k(M)  = \sum_i \langle x_i^{k},[M]\rangle \ . 
$$
By the splitting principle $s_k$ is a linear combination of Chern numbers. Milnor~\cite{M,TM} proved that one can take
the generator $\beta_k$ in degree $2k$ to be $[M]$ if and only if $s_k(M)\neq 0$.

The proof of Theorem~\ref{t:Betti} will also use special sequences of ring generators for $\Omega^U_{\star}\otimes\Q$.
For that proof we will choose generators belonging to families on which certain linear combinations of Chern numbers
are unbounded, although the Betti numbers are fixed.

\section{Some algebraic surfaces with useful properties}

The following theorem is the starting point for the proofs of the first three main theorems of this paper.
\begin{thm}[\cite{MAorient}]\label{t:MA}
There exist pairs $(X,Y)$ of simply connected complex projective surfaces of non-zero signature,
 which are orientation-reversingly homeomorphic with respect to the orientations defined by their
 complex structures.
 
 Moreover, one may take $X$ to be the four-fold blowup $X' \#\blowup \#\blowup\#\blowup\#\blowup$ of 
 some other surface $X'$.
\end{thm}
The main theorem of~\cite{MAorient} in fact provides infinitely many such pairs, at least if one does
not insist on the property that $X$ be a four-fold blowup. The proof combines geography results due to Persson~\cite{P}
in the case of negative signature and Chen~\cite{Chen} and Moishezon and Teicher~\cite{MT} in the case of positive
signature to find candidate pairs whose Chern numbers are related by
    \begin{alignat}{1}
    c_2(Y) &=c_2(X) \label{eq:c2} \\
	c_{1}^{2}(Y) &=4c_2(X)-c_{1}^{2}(X) \ . \label{eq:c1}
	\end{alignat}
These equations are equivalent to requiring that $X$ and $Y$ have the same Euler characteristic and have 
opposite signatures. As soon as both surfaces are simply connected and non-spin, they are orientation-reversingly 
homeomorphic by Freedman's classification result for simply connected four-manifolds~\cite{freed}. 
The geography results used are flexible enough to allow one to make $X$ non-minimal and, in fact a four-fold
blowup. In~\cite[Theorem 3.10]{MAorient} such a result was stated for double blowups, and the case of four-fold 
blowups is exactly the same.

For the construction of diffeomorphic projective varieties of higher dimension, whose differences will be 
used as generators of the rational complex cobordism ring, we shall need the following.
\begin{lem}\label{l:bundles}
Let $(X,Y)$ be a pair of algebraic surfaces as in Theorem~\ref{t:MA}. Then for every $n\geq 1$ there exist
holomorphic vector bundles $\E\longrightarrow X$ and $\F\longrightarrow Y$ of rank $n+1$ with trivial first Chern classes
and 
$$
\langle c_2(\E),[X]\rangle = - \langle c_2(\F),[Y]\rangle \neq 0 \ .
$$
\end{lem}
\begin{proof}
It is enough to prove the case $n=1$, for one can then stabilize by direct summing with trivial line bundles.

For $n=1$ we take $\F=\OO(K_Y)\oplus\OO(-K_Y)$, where $K_Y$ is a canonical divisor of $Y$. Then $\F$ has trivial first Chern class and 
$$
 \langle c_2(\F),[Y]\rangle = (K_Y)\cdot (-K_Y) = - c_1^2(Y) < 0 \ ,
$$
where the final inequality follows, for example, from~\eqref{eq:c1} and the Miyaoka--Yau inequality
$c_1^2(X)\leq 3c_2(X)$ using $c_2(X)>0$.
To prove the Lemma we now have to find a holomorphic rank two bundle $\E\longrightarrow X$ with trivial first Chern class and with 
$$
\langle c_2(\E),[X]\rangle =- \langle c_2(\F),[Y]\rangle=c_1^2(Y) > 0 \ .
$$
Let $E_i$ be the exceptional divisors in $X$.
Since every positive integer is a sum of four squares, we can find integers $a_i$ such that the divisor 
$$
D = a_1E_1+a_2E_2+a_3E_3+a_4E_4
$$
satisfies 
$$
D^2=-a_1^2-a_2^2-a_3^2-a_4^2 = -c_1^2(Y) \ .
$$ 
Now $\E=\OO(D)\oplus\OO(-D)$ has the desired property.
\end{proof}

We shall take projectivisations of such holomorphic bundles to construct high-dimensional examples. 

\begin{rem}
The holomorphic bundles constructed in the Lemma may seem rather arbitrary. For $n\geq 3$ one can take 
instead the stabilized direct sums of the holomorphic tangent and cotangent bundles 
$\E = TX\oplus T^*X\oplus\OO_X^{\oplus (n-3)}$ and $\F = TY\oplus T^*Y\oplus\OO_Y^{\oplus (n-3)}$.
These bundles are in some sense canonical, and have the nice property that their second Chern number is a universal 
multiple of the signature. Moreover, one can use them without arranging one of the surfaces to be a four-fold blowup.
But then one would still need other examples for $n=1$ and $2$.
\end{rem}

\section{Projective space bundles over algebraic surfaces}

Let $S$ be a smooth complex projective surface, and $\E\longrightarrow S$ a holomorphic vector bundle of rank $n+1\geq 2$.
Then the projectivisation $\PP(\E)$ is a $\C P^n$-bundle over $S$ with an induced complex structure on the total space.
We now calculate the Thom--Milnor number $s_{n+2}(\PP(\E))$ for certain special choices of $\E$. 

The cohomology ring of $\PP (\E)$ is described by the following consequence of the Leray--Hirsch theorem and of 
Grothendieck's definition of Chern classes~\cite{G}:
\begin{prop}\label{p:coho}
The integral cohomology ring of $\PP (\E)$ is generated as a $H^*(S)$-module by a class $y\in H^2(\PP (\E))$ that restricts
to every $\C P^{n}$-fibre as a generator, subject to the relation 
$$
y^{n+1}+c_1(\E)y^{n}+c_2(\E)y^{n-1}=0 \ .
$$ 
\end{prop}

The next result provides the only Chern number calculation required for the proofs of our main theorems.
\begin{prop}\label{p:sn}
Let $\E\longrightarrow S$ be a holomorphic vector bundle of rank $n+1$ with $c_1(\E)=0$ and $\langle c_2(\E),[S]\rangle =c \ $. 
Then $s_{n+2}(\PP(\E)) = -(n+1)(n+3)\cdot c\ $.
\end{prop}
\begin{proof}
Let
\begin{alignat*}{1}
c(S) &=(1+x_1)(1+x_2) \\
c(\E) &=(1+y_1)\cdot\ldots\cdot (1+y_{n+1}) 
\end{alignat*}
be formal factorizations of the total Chern classes in the sense of the splitting principle.

Since the projection $\pi\colon\PP(\E)\longrightarrow S$ is holomorphic, we can use the Whitney sum formula to conclude 
$$
c(\PP(\E)) = c(T\pi) \cdot \pi^*c(S) \ ,
$$
where $T\pi$ is tangent bundle along the fibers of $\pi$. In what follows we shall suppress the cohomological pullback $\pi^*$ in 
the notation.

To compute $c(T\pi)$ we use the tautological exact sequence
$$
1\longrightarrow \LL^{-1}\longrightarrow \pi^*\E\longrightarrow \LL^{-1}\otimes T\pi \longrightarrow 1\ ,
$$
where $\LL$ is the fiberwise hyperplane bundle on the total space. Tensoring with $\LL$, we conclude 
$$
c(T\pi) =c(\LL\otimes \pi^*\E) = (1+y+y_1)\cdot\ldots\cdot (1+y+y_{n+1}) \ ,
$$
where the last equality comes from $c_1(\LL)=y$. Thus
$$
c(\PP(\E)) = (1+y+y_1)\cdot\ldots\cdot (1+y+y_{n+1})(1+x_1)(1+x_2)  \ .
$$
Now to calculate $s_{n+2}(\PP(\E))$ we have to evaluate the following expression in the Chern roots on the fundamental class of $\PP(\E)$:
$$
(y_1+y)^{n+2}+\ldots+(y_{n+1}+y)^{n+2}+x_1^{n+2}+x_2^{n+2} \ .
$$
Since $n+2\geq 3$ and symmetric functions of the $x_i$, respectively of the $y_j$, vanish in degrees $\geq 3$,
this expression equals
$$
(n+1)y^{n+2}+(n+2)y^{n+1}\sum_jy_j+{n+2 \choose 2}y^{n}\sum_jy_j^2  \ . 
$$
In the first summand we may, by Proposition~\ref{p:coho}, substitute $y^{n+2}=-c_2(\E)y^n$.
The second summand vanishes since $\sum_jy_j=c_1(\E)=0$. Finally, in the last summand we use
$$
\sum_jy_j^2 = (c_1^2-2c_2)(\E)=-2c_2(\E) \ .
$$
This gives us 
$$
(n+1)y^{n+2}+(n+2)y^{n+1}\sum_jy_j+{n+2 \choose 2}y^{n}\sum_jy_j^2  = -(n+1)(n+3)c_2(\E)y^n \ .
$$
Now the conclusion follows since $y^n$ evaluates as $+1$ on the fibre of $\pi$.
\end{proof}

\begin{rem}
The results of this section are purely
topological, and hold in greater generality, assuming only that $S$ is an almost-complex 
four-manifold and $E$ is some complex vector bundle.
\end{rem}

\section{Proofs of the main theorems}

In this section we first prove Theorems~\ref{t:DOc}, \ref{t:Dc} and~\ref{t:Hc}, which in turn imply Theorems~\ref{t:DO}, \ref{t:D} and~\ref{t:H} 
stated in the introduction. For the proofs we fix a pair of complex projective surfaces $X$ and $Y$ as in Theorem~\ref{t:MA}. They are 
simply connected with non-zero signature, and they are orientation-reversingly homeomorphic.
For every $n\geq 1$ we consider holomorphic bundles $\E_n\longrightarrow X$ and $\F_n\longrightarrow Y$ of rank $n+1$ 
with trivial first Chern classes and with opposite second Chern numbers, compare Lemma~\ref{l:bundles}.

For $k\geq 3$ let $X_k=\PP (\E_{k-2})$ and $Y_k=\PP (\F_{k-2})$. Since $X$ and $Y$ are projective-algebraic the 
holomorphic bundles $\E_{k-2}$ and $\F_{k-2}$ and their projectivisations $X_k$ and $Y_k$ are algebraic as well by 
Serre's GAGA principle~\cite{GAGA}.

We now define the generators for the rational complex cobordism ring that we shall use for the proofs of the first three main theorems.
\begin{defn}
Let $\beta_1=[\C P^1]$, $\beta_2=[X]-[Y]$ and $\beta_k=[X_k]-[Y_k]$ for $k\geq 3$.
\end{defn}
First of all, these are indeed generators.
\begin{prop}
The elements $\beta_1=[\C P^1]$, $\beta_2=[X]-[Y]$ and $\beta_k=[X_k]-[Y_k]$ for $k\geq 3$ form ring generators for $\Omega^U_{\star}\otimes\Q$.
\end{prop}
\begin{proof}
By the result of Milnor~\cite{M,TM} mentioned earlier we only have to check that $s_k(\beta_k)\neq 0$ for all $k$. Clearly there is nothing to 
prove for $k=1$.

For $k=2$ we have 
$$
s_2(\beta_2)=s_2(X)-s_2(Y)=(c_1^2-2c_2)(X)-(c_1^2-2c_2)(Y)=3 (\sigma(X)-\sigma(Y))=6\sigma (X)\neq 0 \ ,
$$
where $\sigma$ denotes the signature. We have used the fact that $X$ and $Y$ have opposite signature since they are orientation-reversingly
homeomorphic, and that these signatures are non-zero.

For $k\geq 3$ Proposition~\ref{p:sn} and Lemma~\ref{l:bundles} give
\begin{alignat*}{1}
s_k(\beta_k)=s_k(X_k)-s_k(Y_k) &=-(k-1)(k+1)(\langle c_2(\E),[X]\rangle -\langle c_2(\F),[Y]\rangle )\\
&= -2(k^2-1)\langle c_2(\E),[X]\rangle\neq 0 \ .
\end{alignat*}
This completes the proof.
\end{proof}

Now we can prove the first three main theorems.

\begin{proof}[Proof of Theorem~\ref{t:Dc}]
Denote by $\bar Y$ the smooth manifold underlying $Y$ equipped with the orientation that is opposite to the 
one defined by the complex structure. By a result of Wall~\cite{W}, the four-manifolds $X$ and $\bar Y$ are smoothly 
$h$-cobordant. Let $W$ be any $h$-cobordism between them. 

The product $\C P^1\times W$ is a $7$-dimensional $h$-cobordism between $\C P^1 \times X$ and $\C P^1\times \bar Y$. 
By Smale's $h$-cobordism theorem~\cite{Smale} these two manifolds are diffeomorphic. This shows that $\beta_1\cdot\beta_2\in\D_{6}$.

Similarly $X\times W$ is an $h$-cobordism between $X\times X$ and $X\times\bar Y$, and $\bar Y\times W$ is an $h$-cobordism
between $\bar Y\times X$ and $\bar Y\times \bar Y$. Therefore, by the $h$-cobordism theorem, $X\times X$, $X\times \bar Y=\bar Y\times X$ and $\bar Y\times\bar Y$ 
are all diffeomorphic. Thus $\beta_2\cdot\beta_2\in\D_8$.

As the inclusion of $X$ into $W$ is a homotopy equivalence, the complex vector bundle $E_{k-2}$ underlying the holomorphic 
bundle $\E_{k-2}$ has a unique extension $\EE_{k-2}\longrightarrow W$. The restriction of $\EE_{k-2}$ to $\bar Y$ is isomorphic to the 
complex vector bundle $F_{k-2}$ underlying $\F_{k-2}$, if we think of $F_{k-2}$ as a bundle over $\bar Y$ rather than over $Y$. 
This follows from the conditions imposed on $\E_{k-2}$ and $\F_{k-2}$ in Lemma~\ref{l:bundles}, since they imply that the two
complex vector bundles have the same Chern classes. A direct obstruction theory argument then shows that they are isomorphic,
a conclusion that can also be reached by appealing to the work of Peterson~\cite{Pet}.
Now the projectivisation $\PP (\EE_{k-2})$ is an $h$-cobordism between $X_k$ and $\bar Y_k$. Applying Smale's $h$-cobordism 
theorem~\cite{Smale} again, we conclude $\beta_k\in\D_{2k}$ for $k\geq 3$.
\end{proof}

\begin{proof}[Proof of Theorem~\ref{t:Hc}]
We have $\beta_2\in\HH_4$ by the assumption that $X$ and $Y$ are homeomorphic, and $\beta_k\in\D_{2k}\subset\HH_{2k}$ for $k\geq 3$
by the proof of Theorem~\ref{t:Dc}. This is all we had to do to prove Theorem~\ref{t:Hc}.
\end{proof}

\begin{proof}[Proof of Theorem~\ref{t:DOc}]
The only thing left to discuss is orientations.

In the proof of Theorem~\ref{t:Dc} we have shown that $X_k$ and $Y_k$ are orientation-reversingly diffeomorphic. 
If the complex dimension $k$ is odd, then the complex structure on $Y_k$ that is the complex conjugate of the given 
one induces the opposite orientation $\bar Y_k$. This conjugate structure has the same Chern numbers and is also 
projective algebraic. Thus, if we complex conjugate the complex structure
of $Y_k$, but not the one on $X_k$, then the modified $\beta_k$ has the property $\beta_k\in\DO_{2k}$ for $k$ odd.

We showed in the proof of Theorem~\ref{t:Dc} that $\C P^1 \times X$ and $\C P^1\times \bar Y$ are orientation-preservingly
diffeomorphic. As $\C P^1$ admits orientation-reversing self-diffeomorphisms, it follows that $\C P^1 \times X$ and $\C P^1\times Y$
are also orientation-preservingly diffeomorphic. This means that $\beta_1\cdot\beta_2\in\DO_6$. The same argument shows that 
$\beta_1\cdot\beta_{k-1}\in\DO_{2k}$ for all $k\geq 3$.
\end{proof}

It remains to prove Theorem~\ref{t:Betti} from the introduction. For this we shall use the following notation.
The Hirzebruch $\chi_y$-genus combines all the $\chi_p$ into the polynomial
$$
\chi_y = \sum_{p=0}^n \chi_p y^p \ .
$$
The compatibility of the Hodge and K\"unneth decompositions for the cohomology of a product shows that $\chi_y$ defines a ring homomorphism
$$
\chi_y\colon\Omega^U_{\star}\otimes\Q\longrightarrow\Q [y] \ ,
$$
whose kernel is an ideal $\I_{\star}\subset \Omega^U_{\star}\otimes\Q$.

\begin{proof}[Proof of Theorem~\ref{t:Betti}]
In complex dimension $1$ there is nothing to prove. In dimension $2$ both Chern numbers are linear combinations of 
$\chi_0=\frac{1}{12}(c_1^2+c_2)$ and $\chi_1=\frac{1}{6}(c_1^2-5c_2)$, so again there is nothing to prove.
For the rest of the proof we work in dimension $i\geq 3$. 
In these dimensions we have to prove that a linear combination of Chern numbers that is not a linear combination of the $\chi_p$
is unbounded on some complex projective varieties with bounded Betti numbers. For this we consider the rational complex 
cobordism ring $\Omega^U_{\star}\otimes\Q$ and choose a convenient
generating sequence $\gamma_i$ by setting $\gamma_1=[\C P^1]$, $\gamma_2 =[\C P^2]$ and $\gamma_i =[\PP(\E_c)]$ for 
$i\geq 3$, where $\E_c\longrightarrow A$ is a holomorphic bundle of rank $i-1$ over an Abelian surface $A$ with
$c_1(\E_c)=0$ and with second Chern number $c\neq 0$. Proposition~\ref{p:sn} shows that this is indeed a sequence 
of ring generators.

For definiteness we may take $\E_c = \HH\oplus\HH^{-1}\oplus\OO_A^{\oplus (i-3)}$,
where $\HH$ is an ample line bundle. Then 
$$
c = (c_1(\HH))\cdot (-c_1(\HH)) = - c_1^2(\HH) < 0 \ .
$$

The  subspace $\I_{2i}\subset \Omega^U_{2i}\otimes\Q$ is the intersection of the kernels of all the $\chi_p$ and is therefore of 
codimension $[(i+2)/2]$. A $\Q$-vector space basis for $\Omega^U_{2i}\otimes\Q$ is given by the elements 
$\gamma_I = \gamma_{i_1}\cdot\ldots\cdot\gamma_{i_l}$, where $I=(i_1,\ldots,i_l)$ ranges over all partitions of $i$. 
Among these basis vectors there are $[(i+2)/2]$ many corresponding to partitions with all $i_j\leq 2$, and these are clearly not 
contained in $\I_{2i}$. However, all the other basis vectors are in the subspace $\I_{2i}$. To check this it is enough to check that 
each $[\PP (\E_c)]$ is in $\I_{2i}$. This follows directly from looking at the Hodge decomposition
of the cohomology of $\PP (\E_c)$. This cohomology is given by Proposition~\ref{p:coho}, with the class $y$ being of type $(1,1)$.
Here we use the assumption that the base of $\E_c$ is an Abelian surface $A$, with $\chi_p(A)=0$ for all $p$.

It follows by counting dimensions that the $\gamma_I$ corresponding to partitions $I$ containing an $i_j\geq 3$ form a vector space 
basis of $\I_{2i}$.

Let $f\colon \Omega^U_{2i}\otimes\Q\longrightarrow\Q$ be any linear combination of Chern numbers. If $f$ is not a linear combination of 
the $\chi_p$, then $\ker (f)\cap \I_{2i}$ is a proper subspace of $\I_{2i}$. It follows that at least one of the $\gamma_I$ with $I$ containing 
an $i_j\geq 3$ is not in $\ker (f)$, i.e. $f(\gamma_I)\neq 0$ for this particular $I$. The projective variety representing $\gamma_I$ fibers
holomorphically over the Abelian surface $A$. Pulling back under a finite covering $A'\longrightarrow A$ of degree $d$ we obtain
a complex projective variety on which $f$ evaluates as $d\cdot f(\gamma_I)$, which is unbounded as we increase $d$. However,
the Betti numbers of these coverings are bounded. This completes the proof.
\end{proof}

Proposition~\ref{p:sn} shows that the number $s_i$ is unbounded on complex projective varieties of dimension $i$ with
fixed Betti numbers. As another concrete instance for the unboundedness of Chern numbers we prove:
\begin{prop}\label{p:c1i}
In every complex dimension $i\geq 3$ there are sequences of complex projective varieties with bounded Betti numbers on which
the Chern number $c_1^i$ is unbounded.
\end{prop}
\begin{proof}
We consider the same $\PP (\E_c)$ with $c\neq 0$ as in the proof of Theorem~\ref{t:Betti}. The formula for the total Chern class from 
the proof of Proposition~\ref{p:sn} shows that $c_1(\PP (\E_c))=(i-1)y$. Thus the relation $y^i = -c_2(\E_c)y^{i-2}$ from 
Proposition~\ref{p:coho} gives
$$
\langle c_1^i(\PP (\E_c)),[\PP (\E)]\rangle = (i-1)^i \langle y^i,[\PP (\E_c)]\rangle =  -(i-1)^i \langle c_2(\E_c),[A]\rangle = -(i-1)^i\cdot c \neq 0 \ .
$$
Pulling back to finite coverings of $A$ shows that $c_1^i$ is unbounded on these coverings, although their Betti numbers are bounded.
\end{proof}

\begin{rem}
I proved in~\cite[Proposition~2]{Chern} that in odd complex dimensions $>1$ the Hirzebruch--Riemann--Roch
formulae for $\chi_p$ do not involve $c_1^i$. This shows of course that $c_1^i$ is not contained in the span of the 
$\chi_p$. However, in even complex dimensions the Todd genus expressing $\chi_0$ does contain $c_1^i$, so that,
without Proposition~\ref{p:c1i} and Theorem~\ref{t:Betti}, one would not know that $c_1^i$ is not contained in the span of the $\chi_p$.
\end{rem}

\bibliographystyle{amsplain}

\begin{thebibliography}{10}

\bibitem{BH}
A.~Borel and F.~Hirzebruch, {\em Characteristic classes and 
homogeneous spaces, I, II}, Amer.~J.~Math.~{\bf 80} (1958), 458--538 and {\bf 81} (1959), 315--382;
reprinted in~\cite{Hir2}.

\bibitem{BuonCri}
S.~Buoncristiano and D.~Hacon, {\em The geometry of Chern numbers}, 
Ann.~Math.~{\bf 118} (1983), 1--7.

\bibitem{Chen}
Z.~Chen, {\em On the geography of surfaces -- simply connected minimal surface of positive index}, 
Math.~Ann.~{\bf 277} (1987), 141--164.

\bibitem{freed}
M.~H.~Freedman, {\em The topology of four--manifolds}, J.~Differential
Geometry {\bf 17} (1982), 357--454.

\bibitem{G}
A.~Grothendieck, {\em La th\'eorie des classes de Chern}, Bull.~Soc.~math.~France {\bf 86} (1958), 137--154.

\bibitem{Hir1}
F.~Hirzebruch, {\em Some problems on differentiable and complex manifolds},
Ann.~Math.~{\bf 60} (1954), 213--236;
reprinted in~\cite{Hir2}.

\bibitem{TMAG}
F.~Hirzebruch,
{\sl Neue topologische Methoden in der algebraischen Geometrie}, 
2.~erg\"anzte Auflage, Springer Verlag 1962;
reprinted in~\cite{Hir2}.

\bibitem{Hir2}
F.~Hirzebruch,
{\sl Gesammelte Abhandlungen}, Band I, Springer-Verlag 1987.

\bibitem{H05}
F.~Hirzebruch, {\em The projective tangent bundles of a complex three-fold}, Pure and Appl.~Math.~Quarterly {\bf 1} (2005),
441--448.

\bibitem{K2}
P.~J.~Kahn, {\em Chern numbers and oriented homotopy type}, Topology {\bf 7} (1968), 69--93.

\bibitem{MAorient}
D.~Kotschick, {\em Orientation--reversing homeomorphisms in surface
geography}, Math. Annalen {\bf 292} (1992), 375--381.

\bibitem{BLMSorient}
D.~Kotschick, {\em Orientations and geometrisations of compact
complex surfaces}, Bull.~London Math.~Soc.~{\bf 29} (1997), 145--149.

\bibitem{Chern}
D.~Kotschick, {\em Chern numbers and diffeomorphism types of projective varieties}, J.~of Topology {\bf 1} (2008), 518--526.

\bibitem{PNAS}
D.~Kotschick, {\em Characteristic numbers of algebraic varieties}, Proc.~Natl.~Acad.~USA  {\bf 106}, no.~25 (2009), 10114--10115.

\bibitem{KT}
D.~Kotschick and S.~Terzi\'c, {\em Chern numbers and the geometry of partial flag manifolds},
Comment.~Math.~Helv.~{\bf 84} (2009), 587--616.

\bibitem{LW}
A.~S.~Libgober and J.~W.~Wood, {\em Uniqueness of the complex
structure on K\"ahler manifolds of certain homotopy types},
J.~Differential Geometry {\bf 32} (1990), 139--154.

\bibitem{M}
J.~W.~Milnor, {\em On the cobordism ring $\Omega^{\star}$ and a complex analogue, Part I}, Amer.~J.~Math.~{\bf 82} (1960), 505--521; 
reprinted in~\cite{MDT}.

\bibitem{MDT}
J.~W.~Milnor, {\sl Collected Papers of John Milnor, III -- Differential Topology}, American Math.~Soc., Providence, R.I.~2007.

\bibitem{MT}
B.~Moishezon and M.~Teicher, {\em Simply-connected algebraic surfaces of positive index}, Invent.~math.~{\bf 89} (1987),
601--643.

\bibitem{NN}
S.~P.~Novikov, {\em Homotopy properties of Thom complexes}, (Russian) Mat.~Sb.~(N.S.) {\bf 57} (1962), 407--442. 

\bibitem{N}
S.~P.~Novikov, {\em Topological invariance of rational classes of Pontrjagin}, (Russian)
Dokl.~Akad.~Nauk SSSR {\bf 163} (1965), 298--300; engl.~transl.~in Soviet Math.~Dokl.~{\bf 6} (1965), 921--923.

\bibitem{P}
U.~Persson, {\em Chern invariants of surfaces of general type}, Compos.~math.~{\bf 43} (1981), 3--58. 

\bibitem{Pet}
F.~P.~Peterson, {\em Some remarks on Chern classes}, Ann.~Math.~{\bf 69} (1959), 414--420.

\bibitem{GAGA}
J-P.~Serre, {\em G\'eom\'etrie alg\'ebrique et g\'eom\'etrie analytique}, Ann.~Inst.~Fourier Grenoble {\bf 6} (1955-1956), 1--42.

\bibitem{Smale}
S.~Smale, {\em On the structure of manifolds},
Amer.~J.~Math.~{\bf 84} (1962), 387--399.

\bibitem{T}
R.~Thom,
{\em Quelques propri\'et\'es globales des vari\'et\'es diff\'erentiables},
Comment.~Math.~Helv.~{\bf 28} (1954), 17--86.

\bibitem{TM}
R.~Thom, {\em Travaux de Milnor sur le cobordisme}, S\'eminaire Bourbaki, Vol.~{\bf 5}, Exp.~No.~180, 169--177, Soc.~Math.~France, Paris, 1995;
reprinted in~\cite{MDT}.

\bibitem{W}
C.~T.~C.~Wall, {\em On simply-connected 4-manifolds}, J.~London Math.~Soc.~{\bf 39} (1964), 141--149.

\end{thebibliography}

\bigskip

\end{document}